\documentclass[11pt]{article}
\usepackage{amsmath,amsfonts,amssymb,amsthm}
\usepackage{amsmath}
\usepackage{amssymb}
\usepackage{amscd}
\usepackage{multicol}
\usepackage{color}
\usepackage{amsmath}
\usepackage{amssymb}
\usepackage{graphicx}
\usepackage{amscd}
\usepackage{color}
\usepackage{picinpar}

\begin{document}
\def\ni{\noindent}
\def\t{\theta}
\def\O{\Omega}
\def\e{\epsilon}
\def\lra{\longrightarrow}
\def\RR{{\mathbb R}}
\def\NR{{\mathbb N}}
\def\ZZ{{\mathbb Z}}
\def\CC{{\mathbb C}}
\def\l{\lambda}
\def\LL{${\cal L}$}
\def\E{{\cal E}}
\def\a{{\alpha}}
\def\A{A_{\a}}
\def\ta{\t^{\a}}
\def\rot{{\rm rot}}

\def\stretchx{\Bumpeq{\!\!\!\!\!\!\!\!{\longrightarrow}}}

\newcommand{\vs}[1]{\vskip #1pt}

\newtheorem{Theorem}{Theorem}[section]
\newtheorem{Definition}[Theorem]{Definition}
\newtheorem{corollary}[Theorem]{Corollary}
\newtheorem{proposition}[Theorem]{Proposition}
\newtheorem{examples}[Theorem]{Esempi}
\newtheorem{example}[Theorem]{Example}
\newtheorem{lemma}[Theorem]{Lemma}
\newtheorem{remark}[Theorem]{Remark}

\catcode`\@=11


   \renewcommand{\theequation}{\thesection.\arabic{equation}}
   \renewcommand{\section}%
   {\setcounter{equation}{0}\@startsection {section}{1}{\z@}{-3.5ex
plus -1ex
    minus -.2ex}{2.3ex plus .2ex}{\Large\bf}}

\title{Asymptotic and chaotic solutions of a singularly perturbed Nagumo-type equation\thanks{Under the auspices of GNAMPA-I.N.d.A.M., Italy.
The second author is supported by the P.R.I.N. Project ''Variational and perturbative aspects of nonlinear diffrential problems''.
}}
\author{Alberto Boscaggin\footnote{Dipartimento di Matematica, Universit\`a di Torino, Via Carlo Alberto 10, 10123 Torino, Italy.
e-mail: \texttt{alberto.boscaggin@unito.it}}
\and
Walter Dambrosio\footnote{Dipartimento di Matematica, Universit\`a di Torino, Via Carlo Alberto 10, 10123 Torino, Italy.
e-mail: \texttt{walter.dambrosio@unito.it}}
\and
Duccio Papini\footnote{Dipartimento di Matematica e Informatica, Universit\`a di Udine, Via delle Scienze 206, 33100 Udine, Italy.
e-mail: \texttt{duccio.papini@uniud.it}}}
\date{}
\maketitle

\begin{abstract}
We deal with the singularly perturbed Nagumo-type equation
$$
\epsilon^2 u'' + u(1-u)(u-a(s)) = 0,
$$
where $\epsilon > 0$ is a real parameter and $a: \mathbb{R} \to \mathbb{R}$ is a piecewise constant function satisfying
$0 < a(s) < 1$ for all $s$. We prove the existence of chaotic, homoclinic and heteroclinic solutions, when
$\epsilon$ is small enough. We use a dynamical systems approach, based on the Stretching Along Paths method and
on the Conley-Wa\.zewski's method.
\end{abstract}

\noindent
{\footnotesize \textbf{AMS-Subject Classification}}. {\footnotesize 34C25; 34C28; 34C37.}

\section{Introduction}
\def\theequation{1.\arabic{equation}}\makeatother
\setcounter{equation}{0}

In this paper, we deal with the singularly perturbed Nagumo-type equation
\begin{equation}\label{eqintro}
\epsilon^2 u'' + u(1-u)(u-a(s)) = 0,
\end{equation}
where $\epsilon > 0$ is a small parameter and
$a: \mathbb{R} \to \mathbb{R}$ is a locally integrable function satisfying
$$
0 < a(s) < 1, \quad \mbox{ for every } s \in \mathbb{R}.
$$
Notice that equation \eqref{eqintro} has the two constant solutions $u \equiv 0$
and $u \equiv 1$ which actually behave like saddle points);
we will be interested in the existence of solutions $u$ satisfying $0 < u(s) < 1$ for all $s$.
\medbreak
Our investigation is motivated by a classical paper by Angenent, Mallet-Paret and Peletier \cite{AngMalPel87},
dealing with the Neumann boundary value problem associated with \eqref{eqintro} in the framework
of steady-states solutions of the corresponding parabolic problem
(which arises in population genetics). On the lines of previous works \cite{ClePel85,Kur83}, in \cite{AngMalPel87}
it is proved that, for $\epsilon$ small enough, the Neumann boundary value problem
associated with \eqref{eqintro} has multiple solutions, whose shape and limit profile (for $\epsilon \to 0^+$)
can be precisely described in terms of the zeros of the function $a - 1/2$.
This analysis suggests that the dynamics of \eqref{eqintro} could be quite rich provided
the function $a$ crosses the value $1/2$.
\medbreak
The aim of the present paper is indeed to prove the existence of solutions to \eqref{eqintro} defined on the whole real line
and exhibiting complex behavior, when $a$ switches infinitely many times between two values $a_-,a_+$ with
$$
0 < a_- < \frac{1}{2} < a_+ < 1.
$$
(see, however, Remark \ref{stabile} for possible generalizations of the above assumption).
More precisely, we will provide solutions of essentially two types. On one hand,
we find the existence of globally defined solutions rotating (in the phase-plane)
a certain number of times around $(a_{\pm},0)$ in suitable intervals $I^{\pm}_j$ ($j \in \mathbb{Z}$)
(solutions of this kind will be called \emph{chaotic});
the number of revolutions can be arbitrarily prescribed and
$\cup_{j \in \mathbb{Z}}I^{\pm}_j$ is unbounded from below and from above, i.e., solutions
oscillate infinitely many times on the whole real line. In particular, if $a$ is $T$-periodic for some $T > 0$,
we find the existence of subharmonic solutions to
\eqref{eqintro} with complex behavior; indeed, one could show that the Poincar\'e map associated with \eqref{eqintro}
on a period is topologically semiconjugated to the Bernoulli shift on a suitable number of symbols (as in \cite{CapDamPap02}, see
also Remark \ref{coniugio}).

On the other hand, we are able to produce \emph{homoclinic} and \emph{heteroclinic} solutions, having the same nodal behavior as before on bounded intervals of
arbitrarily large length and converging monotonically to one of the equilibria $(0,0)$ and $(1,0)$ for
$t \to \pm \infty$.

A similar nonlinearity is considered in the recent papers \cite{EllZan13, EllZan->, ZaZa-12, ZaZa-14}.
However, those papers deal with situations which can be considered in a certain sense dual to ours, since they have a fixed central equilibrium
(a center, in fact) while the two nearby saddle points move according to a suitable piecewise constant weight function.
Moreover, we decided to focus on finding solutions that are obtained by exploiting quite different configurations than those investigated in the
mentioned papers, and to omit the details in cases that are similar to those already treated (see Remark \ref{linkinginterno}).
\medbreak
For the proof of our results, we use the change of variable
$x(t) = u(\epsilon t)$, so that the equivalent planar system associated with \eqref{eqintro} is transformed into
\begin{equation}\label{sysintro2}
x' = y, \qquad  y' = x(x-1)(x-a(\epsilon t));
\end{equation}
hence, the smallness of $\epsilon > 0$ reflects into the fact that the function $a_{\epsilon}(t) = a(\epsilon t)$
is constant on intervals of large amplitude. As a consequence, \eqref{sysintro2} can be regarded as a slowly varying perturbation
of an autonomous system with two hyperbolic equilibrium points; in this setting, the arising of rich dynamics seems to
be a quite common phenomenum (see, among others, \cite{FelMarTan06,GedKokMis02} as well as the bibliography in \cite{ZaZa-14}).
\medbreak
In order to detect such a complex behavior, we use a dynamical systems approach,
based on a careful analysis of the trajectories of the (piecewise autonomous) system \eqref{sysintro2}.
More in detail, we first rely on a topological technique of path stretching (the so-called SAP method) developed in \cite{PapZan04b,PapZan04,
PirZan05,PirZan07} to detect the presence of symbolic dynamics (as well as of periodic points, when $a$ is periodic) for suitable Poincar\'e maps associated with \eqref{sysintro2}. As a by-product, this approach also provides us planar paths which can be used
to connect the stable and unstable manifolds of the equilibria so as to
obtain heteroclinic and homoclinic solutions with a complex nodal behavior (see, among others, \cite{EllZan13,Gav11,HolStu92,MarRebZan09}
and the references therein for related results in this direction). It has to be noticed that, since the equation is non-autonomous,
the existence of the stable/unstable manifolds is not straightforward; we use indeed the classical Conley-Wa\.zewski's method \cite{Con75,Waz47}
(see also \cite{EllZan->,PapZan00}) to prove that these sets actually exist and can be suitably localized.
\medbreak
The plan of the paper is the following. In Section \ref{sec2}, we briefly discuss the autonomous case.
In Section \ref{sec3}, we prove some stretching properties, as well as the existence of stable and unstable manifolds to
the equilibria. In Section \ref{sec4}, we state and prove our main results. Finally, some basic facts about SAP method are collected in
a final Appendix.

\section{The autonomous case}\label{sec2}
\def\theequation{2.\arabic{equation}}\makeatother
\setcounter{equation}{0}

In this section we collect some basic results for the autonomous equation
\begin{equation}\label{eqaut}
x'' + x(1-x)(x - a) = 0,
\end{equation}
where $a$ is a real constant such that
$$
0 < a < 1.
$$
Precisely, we are going to perform a phase-plane analysis for the equivalent planar system
\begin{equation}\label{sysaut}
\left\{
\begin{array}{l}
\vspace{0.1cm}
x' = y \\
y'= x(x-1)(x-a).
\end{array}
\right.
\end{equation}
We will always confine our attention to the dynamics in the vertical strip $[0,1] \times \mathbb{R}$.
\smallbreak
It is immediately seen that
the points $(0,0)$, $(a,0)$ and $(1,0)$ are the only equilibrium points of
\eqref{sysaut}. To proceed further, we define the function
$$
\mathcal{E}_a(x,y) = \frac{1}{2}y^2 + F_a(x),\quad \forall \ (x,y)\in \RR^2,
$$
where
$$
F_a(x) = -\frac{1}{4}x^4 + \frac{1+a}{3} x^3 - \frac{a}{2}x^2,\quad \forall \ x\in \RR.
$$
As well-known, system \eqref{sysaut} is conservative and
the function $t \mapsto \mathcal{E}_a(x(t),y(t))$
is constant along solutions $(x(t),y(t))$ to \eqref{sysaut}.
To describe the global dynamics of \eqref{sysaut}, we can thus study the geometry
of the level sets $\mathcal{E}_a^{-1}(c)$ for different values of $c \in \mathbb{R}$.
\smallbreak
It turns out that the value of the constant $a$
plays a significant role. We start by analyzing the case
\begin{equation}\label{a=12}
a = \frac{1}{2},
\end{equation}
which indeed gives rise to the simplest picture.
Precisely, the level set $\mathcal{E}_a^{-1}(c)$ can here be described as follows:
\begin{itemize}
\item[-] for $c = \tfrac{a-2}{12}a^2$, it is the point $(a,0)$;
\item[-] for $\tfrac{a-2}{12}a^2 < c < 0$, it is a closed cycle around $(a,0)$;
\item[-] for $c = 0$, it is the union of the points $(0,0)$, $(1,0)$ and of the heteroclinic orbits joining them;
\item[-] for $c > 0$, it is the union of two curves, one in the half-plane $\{(x,y)\in \RR^2:\ y>0\}$, one in the half-plane $\{(x,y)\in \RR^2:\ y<0\}$, connecting two points of the form $(0,y_1)$ and $(1,y_1)$ for some $y_1 \in \RR$.
\end{itemize}
The phase-portrait is shown in Figure \ref{fig1}.

\begin{figure}[!h]
\centering
\includegraphics[height=8cm,width=10cm]{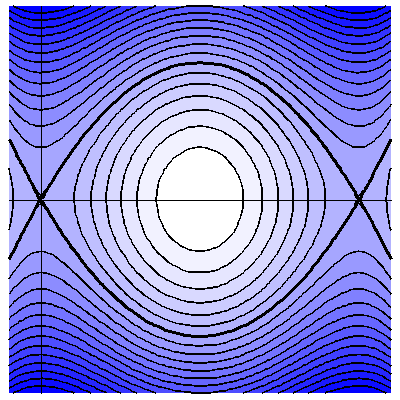}
\caption{\small{The phase-portrait of the autonomous system \eqref{sysaut} for $a = 1/2$.
The heteroclinic orbits connecting the equilibria $(0,0)$ and $(1,0)$
are painted with a darker color. For graphical reasons, a slightly different $x$ and $y$ scaling has been used.}}
\label{fig1}
\end{figure}

We now turn our attention to the case
\begin{equation}\label{a<12}
0 < a < \frac{1}{2};
\end{equation}
now, for the level set $\mathcal{E}_a^{-1}(c)$ we have the following:
\begin{itemize}
\item[-] for $c = \tfrac{a-2}{12}a^2$, it is the point $(a,0)$;
\item[-] for $\tfrac{a-2}{12}a^2 < c < 0$, it is a closed cycle around $(a,0)$;
\item[-] for $c = 0$, it is the union of the point $(0,0)$ and its homoclinic orbit $H(a)$; for further convenience, we denote by
$(z_a,0)$ the point of intersection between $H(a)$ and the positive $x$-semiaxis;
\item[-] for $0 < c < \tfrac{1-2a}{12}$, it is made by a curve lying between the homoclinic to $(0,0)$ and the stable/unstable manifold $H^{\pm} (a)$ of $(1,0)$, connecting a point of the form $(0,y_1)$ with a point of the form $(0,-y_1)$, for some
$y_1 > 0$;
\item[-] for $c = \tfrac{1-2a}{12}$, it is the union of the point $(1,0)$ and its stable/unstable manifolds;
\item[-] for $c > \tfrac{1-2a}{12}$, it is the union of two curves, one in the half-plane $\{(x,y)\in \RR^2:\ y>0\}$, one in the half-plane $\{(x,y)\in \RR^2:\ y<0\}$, connecting a point of the form $(0,y_1)$ with a point of the form $(1,y_2)$, for some $y_1, y_2\in \RR$.
\end{itemize}

Finally, for
\begin{equation}\label{a>12}
\frac{1}{2} < a < 1
\end{equation}
the phase-portrait can be obtained from the previous one with a symmetry with respect to the line $x = \tfrac{1}{2}$. The homoclinic orbit and the stable/unstable manifolds are defined in an analogous way, by swapping $(0,0)$ and $(1,0)$, and will be again denoted by $H(a)$ and $H^\pm (a)$; moreover, $(z_a,0)$ will be the point of intersection between $H(a)$ and the positive $x$-semiaxis.
Both the phase-portraits are shown in Figure \ref{fig2}.

\begin{figure}[!h]
\includegraphics[height=5cm,width=7cm]{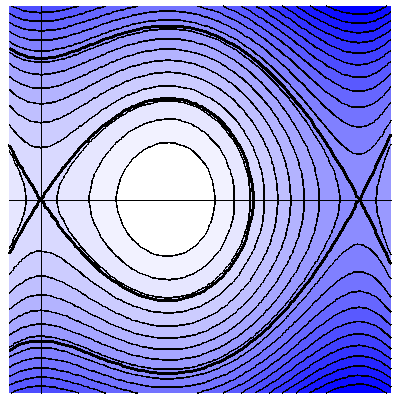}
\hfill
\includegraphics[height=5cm,width=7cm]{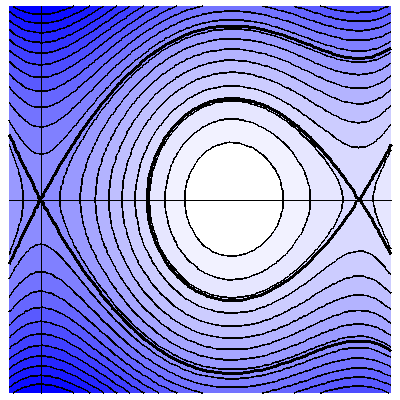}
\caption{\small{On the left, the phase-portrait of the autonomous system \eqref{sysaut} for $0 < a < 1/2$; on the right, the phase-portrait of the autonomous system \eqref{sysaut} for $1/2 < a < 1$.  The homoclinic orbits, as well as the stable/unstable
manifolds, are painted with a darker color. For graphical reasons, a slightly different $x$ and $y$ scaling has been used.}}
\label{fig2}
\end{figure}

\begin{remark}\label{bilanciato}
\textnormal{
From the above discussion, it appears that the phase-portrait of system \eqref{sysaut} is completely different in the case
$a = 1/2$ and in the case $a \neq 1/2$. Indeed, the heteroclinc orbit connecting $(0,0)$ and
$(1,0)$ for $a = 1/2$ disappear as soon as $a \neq 1/2$, splitting into orbits of different type.
In terms of the potential $F_a(x)$, we have indeed $F_a(0) = F_a(1)$ if and only if $a = 1/2$; in this case, the potential is said to be
balanced. In this context, \eqref{eqmain} can be framed in the setting of equations with unbalanced
potentials (compare with the introduction in \cite{NakTan03}).}
\end{remark}

\section{Topological lemmas}\label{sec3}
\def\theequation{3.\arabic{equation}}\makeatother
\setcounter{equation}{0}

In this section we collect the preliminary technical lemmas which will be
used in the proof of our main results.

\subsection{Stretching properties}\label{stre}

In this section, we fix two reals constants $a_-,a_+$ satisfying
$$
0 < a_- < \frac{1}{2} < a_+ < 1.
$$
Our goal is to prove some results for the dynamics of the autonomous system
\eqref{sysaut} for $a = a_-$ and $a= a_+$ on suitable sets which will be constructed below.
Throughout this section, we always refer to the definitions given in the Appendix. Also, to simplify the notation, from now on we denote by $S(a_-)$ (resp., $S(a_+)$) the planar system
\eqref{sysaut} for $a = a_-$ (resp., $a = a_+$); moreover, let
$\Theta(a,z)$ be the orbit of \eqref{sysaut} passing through the point
$z \in \mathbb{R}^2$.

For $T>0$ we finally define the maps
$\Psi_-^T$ and $\Psi_+^T$ as the (restriction to the vertical strip $[0,1] \times \mathbb{R}$) of Poincar\'e maps associated with systems
$S(a_-)$ and $S(a_+)$, respectively, on the interval
$[0,T]$, i.e.
$$
\Psi_\pm^T(x_0,y_0) = (x(T;x_0,y_0),y(T;x_0,y_0)),\quad \forall \, (x_0,y_0) \in [0,1] \times \mathbb{R},
$$
where $(x(\cdot;x_0,y_0),y(\cdot;x_0,y_0))$ is the unique solution
to $S(a_\pm)$ satisfying the initial condition $(x(0),y(0)) = (x_0,y_0)$.
Since we will be interested in the dynamics on the strip $[0,1] \times \mathbb{R}$ only, we can assume that such maps are globally defined just
by suitably modifying the nonlinearity $f_a(x) = x(1-x)(x-a)$ for $x \notin [0,1]$ (for instance, by setting $f_a(x) = 0$
for $x \notin [0,1]$).
\smallbreak
Let us fix $p_-$ and $p_+$ such that
\begin{equation}\label{ordine0}
\frac{1}{2} < \max\left\{z_{a_-},a_+ \right\} < p_- < 1 \qquad \mbox{ and } \qquad 0 < p_+ < \min\left\{a_-,z_{a_+}\right\} < \frac{1}{2}.
\end{equation}
Let $\mathcal{R}_1$ and $\mathcal{R}_3$ be the two connected components of the intersection of the following two strips:
\begin{itemize}
\item[-] the strip $S_-$ between the stable manifold $H^+(a_-)$
and the orbit $\Theta (a_-,(p_-,0))$ of $S(a_-)$,
\item[-] the strip $S_+$ between the unstable manifold $H^+(a_+)$
and the orbit $\Theta (a_-,(p_+,0))$ of $S(a_+)$.
\end{itemize}
Notice that, since $p_+ < z_{a_+}$ and $z_{a_-} < p_-$,
the above defined orbits passing through
$(p_{\pm},0)$ lie between the homoclinics and the stable/unstable manifolds of the corresponding systems
$S(a_\pm)$. Therefore, the defined regions $\mathcal{R}_1$ and $\mathcal{R}_3$ are topological rectangles, and we name $\mathcal{R}_1$ (resp., $\mathcal{R}_3$) the one contained in the upper (resp., lower) half-plane.
\smallbreak
In what follows we also need to provide an orientation for the above constructed rectangles; to this aim, we denote by
$\mathcal{R}_1^\pm$ the components of the boundary of $\mathcal{R}_1$ lying on orbits of system
$S(a_\pm)$ and by $\mathcal{R}_3^\pm$ the components of the boundary of $\mathcal{R}_1$ lying on orbits of system
$S(a_\mp)$.
\medbreak
Now, let $\mathcal{R}_2$ be the intersection of the following
regions:
\begin{itemize}
\item[-] the strip $S_-$ defined above,
\item[-] the annular region between the homoclinic $H(a_+)$ and the closed orbit $\Theta(a_+,(q_+,0))$ of $S(a_+)$, for some $q_+$ such that
\begin{equation}\label{ordine}
p_- < q_+ < 1,
\end{equation}
\item[-] the upper half-plane.
\end{itemize}
Straightforward computations show that the homoclinic $H(a_+)$ is contained in the region bounded
by the stable and unstable manifolds $H^\pm (a_-)$. Moreover, since $a_+ < p_- < q_+$ (recall both \eqref{ordine0} and \eqref{ordine})
the closed orbit $\Theta(a_+,(q_+,0))$
winds around the point $(p_-,0)$, as well. These facts together guarantee that $\mathcal{R}_2$ is a topological rectangle.
\smallbreak
In a similar way, we can define the rectangle
$\mathcal{R}_4$ as the intersection of the following
regions:
\begin{itemize}
\item[-] the strip $S_+$ defined above,
\item[-] the annular region between the homoclinic $H(a_-)$ and the closed orbit $\Theta(a_-,(q_-,0))$ of $S(a_-)$, for some $q_-$ such that
\begin{equation}\label{ordine1}
0 < q_- < p_+,
\end{equation}
\item[-] the lower half-plane.
\end{itemize}
Notice that the rectangles $\mathcal{R}_2$ and $\mathcal{R}_4$ depend also on the choice of the numbers
$q_{\pm}$ satisfying \eqref{ordine} and \eqref{ordine1} which is not arbitrary and will be specified in
Proposition \ref{stretching}.
\smallbreak
We can again orientate the above constructed rectangles, denoting by $\mathcal{R}_2^-$ the components of the boundary of $\mathcal{R}_2$ lying on orbits of the systems $S(a_+)$ and by $\mathcal{R}_2^+$ the remaining components; finally,
$\mathcal{R}_4^-$ are the components of the boundary of $\mathcal{R}_4$ lying on orbits of the systems $S(a_-)$
and $\mathcal{R}_4^+$ are the remaining components.
\medbreak
We illustrate the whole construction of the rectangles $\mathcal{R}_i$, for $i=1,\ldots,4$ in Figure \ref{figrett}.
\begin{figure}[!h]
\centering
\includegraphics[height=10.5cm,width=12.5cm]{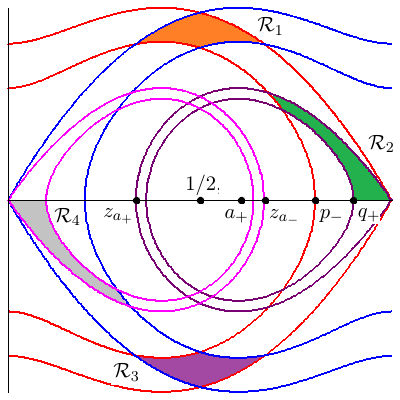}
\caption{\small{The construction of the topological rectangles $\mathcal{R}_i$, for $i=1,\ldots,4$; in particular,
we comment in detail the construction of $\mathcal{R}_1$ and $\mathcal{R}_2$.
Consider the stable manifold (to $(1,0)$) $H^+(a_-)$ (in red), as well
as the homoclinic (to $(0,0)$) $H(a_-)$ (in pink) and recall that such an orbit intersects the
positive $x$-semiaxis at the point $(z_{a_-},0)$. The condition
$p_- > z_{a_-}$ (see \eqref{ordine0}) then guarantees that the piece of orbit of $\Theta (a_-,(p_-,0))$ (also in red) in the upper half-plane lies between
$H(a_-)$ and $H^+(a_-)$. With an analogous construction for the system $S(a_+)$,
and taking into account that $p_+ < 1/2 < p_-$ (see \eqref{ordine0} again) we can thus determine the rectangle
$\mathcal{R}_1$ (in orange) in the upper-half plane (as well as the rectangle $\mathcal{R}_3$ (in purple) in the lower half-plane.
Now, consider the homoclinic (to $(1,0)$) $H(a_+)$ (in black), intersecting the positive $x$-semiaxis
at the point $(z_{a_+},0)$. It is easy to check that this orbit is contained in the region bounded
the stable and unstable manifold $H^\pm (a_-)$; moreover, the condition $z_{a_+} < a_+ < p_-$ (see \eqref{ordine})
implies that it intersects the orbit $\Theta (a_-,(p_-,0))$. Finally, focus on the orbit $\Theta(a_+,(q_+,0))$
(also in black). Since $a_+ < p_- < q_+$ (see both \eqref{ordine0} and \eqref{ordine}), such an orbit
intersects $\Theta (a_-,(p_-,0))$ as well (that is, the closed orbit $\Theta(a_+,(q_+,0))$
winds around the point $(p_-,0)$). This determines the rectangle $\mathcal{R}_2$ (in green). An analogous construction
gives $\mathcal{R}_4$ (in grey). Recall that, for the validity of the stretching properties in Proposition
\ref{stretching}, $q_{\pm}$ cannot be arbitrary numbers satisfying \eqref{ordine} and \eqref{ordine1},
but they have to fulfill further conditions (see the proof).}}
\label{figrett}
\end{figure}

We are now in position to prove our crucial result on the stretching:

\begin{proposition}\label{stretching}
The following stretching properties hold true.
\begin{itemize}
\item[1.] There exists $T_1^* > 0$ such that, for every $T_1 > T_1^*$, we have
$$
\Psi_-^{T_1}: (\mathcal{R}_1,\mathcal{R}_1^-) \,\stretchx\, (\mathcal{R}_2,\mathcal{R}_2^-),
\qquad
\Psi_-^{T_1}: (\mathcal{R}_2,\mathcal{R}_2^+) \,\stretchx\, (\mathcal{R}_3,\mathcal{R}_3^-),
$$
$$
\Psi_+^{T_1}: (\mathcal{R}_3,\mathcal{R}_3^-) \,\stretchx\, (\mathcal{R}_4,\mathcal{R}_4^-),
\qquad
\Psi_+^{T_1}: (\mathcal{R}_4,\mathcal{R}_4^+) \,\stretchx\, (\mathcal{R}_1,\mathcal{R}_1^-),
$$
for a suitable choice of $q^+$ and $q^-$ (close enough to $1$ and to $0$, respectively).
\item[2.] For any $N \in \mathbb{N}$, there exists $T_2^*(N) > 0$ such that, for
every $T_2 > T_2^*(N)$, we have
$$
\Psi_+^{T_2}: (\mathcal{R}_2,\mathcal{R}_2^-) \,\stretchx^N\, (\mathcal{R}_2,\mathcal{R}_2^+),
\qquad
\Psi_-^{T_2}: (\mathcal{R}_4,\mathcal{R}_4^-) \,\stretchx^N\, (\mathcal{R}_4,\mathcal{R}_4^+).
$$
\end{itemize}
\end{proposition}

\begin{proof} 1. We give the details for the stretching property
$$
\Psi_-^{T_1}: (\mathcal{R}_1,\mathcal{R}_1^-) \,\stretchx \, (\mathcal{R}_2,\mathcal{R}_2^-),
$$
the other three statements of the first group being analogous.
We first observe that the time needed to cover the piece of the orbit $\Theta(a_-,(p_-,0))$ contained in the right half-plane is given by
$$
\sqrt{2} \int_0^{p_-} \frac{dx}{\sqrt{F_{a_-}(p_-)-F_{a_-}(x)}};
$$
similarly, the time needed to cover the piece of the orbit $\Theta(a_+,(p_+,0))$ contained in the right half-plane is given by
$$
\sqrt{2} \int_{p_+}^1 \frac{dx}{\sqrt{F_{a_+}(p_+)-F_{a_+}(x)}}.
$$
Now, let
$$
T_1^* = \sqrt{2}  \max\left\{\int_0^{p_-} \frac{dx}{\sqrt{F_{a_-}(p_-)-F_{a_-}(x)}}, \int_{p_+}^1 \frac{dx}{\sqrt{F_{a_+}(p_+)-F_{a_+}(x)}}\right\}
$$
and fix $T_1 > T_1^*$. Moreover, we define
$$
q_+ = \sup \left\{ x \in [p_-,1] : \, (\Psi_-^{T_1})^{-1}(x,0) \in \mathcal{R}_1 \right\}.
$$
Let us observe that the above set is non-empty since, in view of the choice of $T_1$,
$x \mapsto (\Psi_-^{T_1})^{-1}(x,0)$ ($x \in [p_-,1]$) parameterizes a curve joining a point in the second quadrant with
the point $(1,0)$, whose intersection in the first quadrant is contained in $S^-$. By a continuity argument, moreover,
we can easily prove that $p_- < q_+ < 1$. In this manner, \eqref{ordine} is satisfied and this completely determines the rectangle
$\mathcal{R}_2$.
This choice of $ q_{+} $ helps in controlling the behavior of the solutions which move according to system $ S( a_{-} ) $
on a time interval of length $ T_{1} $ and either start from $ \mathcal{R}_{1} $ or arrive on $ \mathcal{R}_{3} $.
In fact, such solutions evolve inside the strip $ S^{-} $ and, by the choice of $ q_{+} $, they can cross the positive $ x $ axis only at a point
lying between $ ( p_{-}, 0 ) $ and $ ( q_{+}, 0 ) $.
Similarly, we define
\[
q_- = \inf \left\{ x \in [ 0, p_{+} ] : \, (\Psi_{+}^{T_1})^{-1}(x,0) \in \mathcal{R}_3 \right\}.
\]
in order to complete the definition of $ \mathcal{R}_{4} $.
\vs{6}
\ni
We now pass to the verification of the stretching property, according to Definition \ref{stretchdef}. Let $\gamma: [0,1] \to \mathcal{R}_1$
be a path such that $\gamma(0)$ and $\gamma(1)$ lies on different components of $\mathcal{R}_1^-$ and, just to fix the ideas,
suppose that $\gamma(1)$ lies on the stable manifold $H^+(a_-)$.

\ni
We observe that $\Psi_-^{T_1}(\gamma(1))$ remains on $H^+(a_-)$ in the first quadrant while, by the choice of $T_1$, $\Psi_-^{T_1}(\gamma(0))$ lies
in the third quadrant; hence, by the positive invariance of the strip $S_-$ for the flow of $S(a_-)$,
we have
$$
\Psi_-^{T_1}(\gamma([0,1])) \cap \mathcal{R}_2 \neq \emptyset.
$$
On the other hand, by the choices of $q_+$ and $T_1$, $\Psi_-^{T_1}(\gamma([0,1]))$ cannot intersect
the boundary of $\mathcal{R}_2$ on the $x$-axis. As a consequence, it intersects the two
components of $\mathcal{R}_2^-$, say $\mathcal{R}_{2,u}^-$ and $\mathcal{R}_{2,d}^-$
($d$ means ``down'' and $u$ means ``up'', with obvious meaning).
Now, let
$$
t_\gamma^- = \sup \left\{ t \in [0,1] : \, \Psi_-^{T_1}(\gamma(t)) \in \mathcal{R}_{2,d}^- \right\}
$$
and
$$
t_\gamma^+ = \inf \left\{ t > t_\gamma^- : \, \Psi_-^{T_1}(\gamma(t)) \in \mathcal{R}_{2,u}^- \right\}.
$$
By construction, the sub-path $\gamma_1 = \Psi_-^{T_1} \circ \gamma|_{[t_\gamma^-,t_\gamma^+]}$
has image contained in $\mathcal{R}_2$ and crosses it from
$\mathcal{R}_{2,d}^-$ to $\mathcal{R}_{2,u}^-$, as desired.
\vs{12}
\ni
2. We now turn our attention to the second group of statements, proving that
$$
\Psi_+^{T_2}: (\mathcal{R}_2,\mathcal{R}_2^-) \,\stretchx^N\, (\mathcal{R}_2,\mathcal{R}_2^+)
$$
(the other one being analogous); we observe that here $\mathcal{R}_2$ appears with different orientations when considered as the domain
or the target space of $\Psi_+^{T_2}$. We introduce the polar coordinate system, centered at $(a_+,0)$,
$$
x = a_+ + r \cos \theta, \qquad y = -r \sin \theta;
$$
then, for any $z \in \mathcal{R}_2$, we denote by
$\theta(\cdot;z)$ the (unique) continuous angular function associated
to the solution of $S(a_+)$ starting from $z$ at the time $t=0$, and such that $\theta(0;z) \in [-\pi,0]$.
Moreover, we recall that the time needed by any solution to system $S(a_+)$ starting from
a point on $\mathcal{R}_{2,d}^-$ to perform one turn around $(a_+,0)$ can be computed as
$$
\sqrt{2} \int_{\eta^+}^{q^+} \frac{dx}{\sqrt{F_{a^+}(q^+)-F_{a^+}(x)}},
$$
where $\eta^+ \in (0,q^+)$ is the only number such that $F_{a^+}(\eta^+) = F_{a^+}(q^+)$.
\vs{6}
\ni
Given the natural number $N \geq 1$, we define
$$
T_2^*(N) = (N+1)\sqrt{2} \int_{\eta^+}^{q^+} \frac{dx}{\sqrt{F_{a^+}(q^+)-F_{a^+}(x)}}
$$
and let $T_2 > T_2^*$. For any $j=1,\ldots,N$ we set
$$
\mathcal{H}_j = \{ z \in \mathcal{R}_2 : \, \theta(T_2,z) \in [(2j-1)\pi,2j\pi] \};
$$
let us observe that the sets $\mathcal{H}_1,\ldots,\mathcal{H}_N$ are compact and disjoint.
\vs{6}
\ni
Now, let $\gamma: [0,1] \to \mathcal{R}_2$ a path such that $\gamma(0)$ and $\gamma(1)$
lie on different components of $\mathcal{R}_2^-$ and, to fix the ideas, assume that $\gamma(1)$ lies on
$\mathcal{R}_{2,u}^-$ (that is, on the stable manifold to $(1,0)$).
Hence, $\theta(T_2,\gamma(1)) < 0$ while, in view of the choice of
$T_2$, $\theta(T_2,\gamma(0)) > 2N\pi$. As a consequence, we can find $N$-disjoint subintervals
$I_1,\ldots,I_N \subset [0,1]$ such that $\gamma(s) \in \mathcal{H}_j$ for any $s \in I_j$
and $\Psi_+^{T_2} \circ \gamma|_{I_j}$
crosses $\mathcal{R}_2$ from one component of $\mathcal{R}_{2}^+$ to the other component
of $\mathcal{R}_{2}^+$.
\end{proof}

\begin{remark} \label{nodali}
\textnormal{We observe that the stretching relationships proved in point 2. of Proposition \ref{stretching}
naturally provide information about some nodal properties of the solutions of the systems $S(a_{\pm})$.
In particular, the (clockwise) winding number around the point $(a_+,0)$ of solutions of $S(a_+)$ starting in
$\mathcal{H}_j$ (for some $j=1, \ldots, N$), in a time interval of length $T_2$, lies in the interval $(j - 1/2,j + 1/2)$
and, more precisely, the first derivative of those solutions vanishes exactly $ 2 j $ times.
Similar considerations holds of course for solutions to
$S(a_-)$ around the point $(a_-,0)$. With a slight abuse of terminology, in our main results we will say that, in this case, solutions
make $j$ turns around $(a_{\pm},0)$.}
\end{remark}

\begin{figure}[!h]
\centering
\includegraphics[height=6cm,width=13cm]{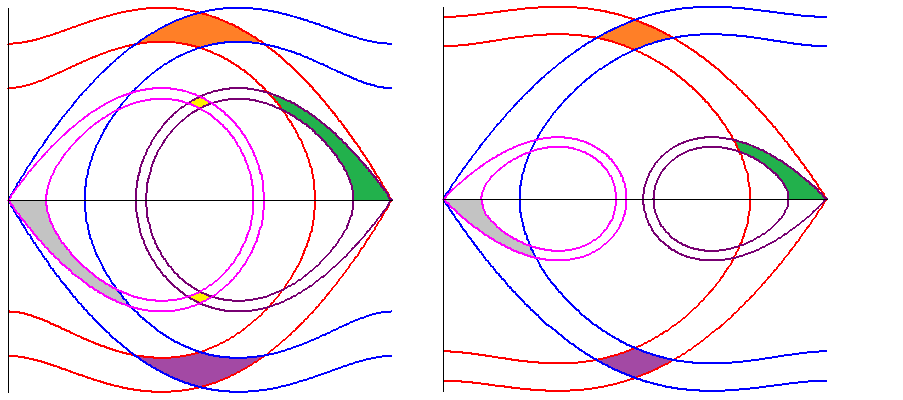}
\caption{\small{On the left, the construction of the rectangles when $a_- = 0.4$ and $a_+ = 0.6$,
so that the homoclinics intersect (this is indeed the same picture as in Figure \ref{figrett})}.
On the right, the construction of the rectangles in a situation in which the homoclinics do not intersect; here
$a_- = 0.3$ and $a_+ = 0.7$. One can immediately realize that nothing changes.
However, on the left again, it is shown how the linking between the homoclinics provides two further topological rectangles,
painted in yellow. It is well-known that such rectangles verify stretching properties, within a Linked Twist Maps
framework (see \cite{MarRebZan10}).}
\label{figltm}
\end{figure}

\begin{remark}\label{LTM}
\textnormal{We observe that the construction of the rectangles $\mathcal{R}_i$ ($i=1,\ldots,4$)
given in this section is independent from the fact that the homoclinics $H(a_-)$ or $H(a_+)$ intersect,
that is, we are not assuming that $z_{a_+} \leq z_{a_-}$ (see Figure \ref{figltm}).
However, it is worth noticing that when the strict inequality
\begin{equation}\label{link}
z_{a_+} < z_{a_-}
\end{equation}
holds true, one can easily construct two further rectangles satisfying stretching conditions.
This is indeed a well-known geometrical configuration, related to the concept Linked Twist Map,
which has been extensively discussed in \cite{MarRebZan10}. We refer again to Figure \ref{figltm}.
With elementary computations, it is possible to show that condition \eqref{link} is satisfied when $a_{\pm}$ are not too far
from the value $1/2$.}
\end{remark}

\subsection{Stable and unstable manifolds}

In this section we prove the existence of stable and unstable manifolds to the equilibria
of a nonautonomous planar system. Our result will be given in a slightly more general setting
than needed; more precisely, we deal with the planar system
\begin{equation}\label{sysq}
\left\{
\begin{array}{l}
\vspace{0.1cm}
x' = y \\
y'=  x(x-1)(x-q(t)),
\end{array}
\right.
\end{equation}
with $q: \mathbb{R} \to \mathbb{R}$ a locally integrable function.
We use the notation $z(t;p,t_0)$ for the (unique) solution to \eqref{sysq}
satisfying the condition $z(t_0) = p$.

\begin{proposition}\label{waze}
The following statements hold true.
\begin{enumerate}
\item Assume that
$$
0 < a_- \leq q(t) \leq a_+ < 1, \quad \mbox{ for a.e. } t \in (-\infty,t_0],
$$
for some $t_0 \in \mathbb{R}$. Then there exist two continua (i.e., connected compact sets)
$\Gamma_{-\infty}^0, \Gamma_{-\infty}^1 \subset \mathbb{R}^2$, lying between the unstable manifolds to $(0,0)$ and $(1,0)$
of the systems $S(a_-)$ and $S(a_+)$ respectively, and such that,
for any $p \in \Gamma_{-\infty}^i$, it holds that
$$
\lim_{t \to -\infty} z(t;p,t_0) \to (i,0), \quad \mbox{ for } i =0,1.
$$
Moreover, there exist two constants $a_-^0 \in (0,a_-)$ and
$a_+^1 \in (a_+,1)$ such that
$$
\Gamma_{-\infty}^0 \cap \left([0,a_-^0] \times \mathbb{R} \right) =
\{(x,y_{-\infty}^0(x)) : \, x \in [0,a_-^0] \}
$$
and
$$
\Gamma_{-\infty}^1 \cap \left([a_+^1,1] \times \mathbb{R} \right) = \{(x,y_{-\infty}^1(x)) : \, x \in [a_+^1,1] \},
$$
for suitable continuous functions
$y_{-\infty}^0:[0,a_-^0] \to \mathbb{R}^+$ and $y_{-\infty}^1:[a_+^1,1] \to \mathbb{R}^-$.
\item Assume that
$$
0 < a_- \leq q(t) \leq a_+ < 1, \quad \mbox{ for a.e. } t \in [t_0,+\infty),
$$
for some $t_0 \in \mathbb{R}$. Then there exist two continua (i.e., connected compact sets)
$\Gamma_{+\infty}^0, \Gamma_{+\infty}^1 \subset \mathbb{R}^2$, lying between the stable manifolds to $(0,0)$ and $(1,0)$
of the systems $S(a_-)$ and $S(a_+)$ respectively, and such that,
for any $p \in \Gamma_{+\infty}^i$, it holds that
$$
\lim_{t \to +\infty} z(t;p,t_0) \to (i,0), \quad \mbox{ for } i =0,1.
$$
Moreover, there exist two constants $a_-^0 \in (0,a_-)$ and
$a_+^1 \in (a_+,1)$ (the same ones as in the previous statement) such that
$$
\Gamma_{+\infty}^0 \cap \left([0,a_-^0] \times \mathbb{R} \right) =
\{(x,y_{+\infty}^0(x)) : \, x \in [0,a_-^0] \}
$$
and
$$
\Gamma_{+\infty}^1 \cap \left( [a_+^1,1] \times \mathbb{R} \right) = \{(x,y_{+\infty}^1(x)) : \, x \in [a_+^1,1] \},
$$
for suitable continuous functions
$y_{+\infty}^0:[0,a_-^0] \to \mathbb{R}^-$ and $y_{+\infty}^1:[a_+^1,1] \to \mathbb{R}^+$.
\end{enumerate}
\end{proposition}

\begin{proof}
We observe that Statement 2 is a consequence of Statement 1 and the symmetry enjoyed by the vector field in \eqref{sysq}
with respect to the axis $ y= 0 $.
Hence we have that $ \Gamma^{i}_{+\infty} = \{ (x,y) : (x,-y) \in \Gamma^{i}_{-\infty}\} $ and $ y^{i}_{+\infty} \equiv -y^{i}_{-\infty} $
for $ i = 0, 1 $.
Concerning the proof of Statement 1, we give only the details about the continuum $\Gamma_{-\infty}^0$, since the existence and the
properties of $ \Gamma_{-\infty}^{1} $ can be deduced in an analogous way.

Let us fix $\epsilon > 0$ such that
$$
0 < a_- -\epsilon < a_+ +\epsilon < 1
$$
and define $\Sigma_\epsilon \subset \mathbb{R}^2$ as the compact triangular region of the first quadrant
of the phase-plane bounded by the unstable manifolds to
$(0,0)$ for the systems $S(a_- -\epsilon)$ and $S(a_+ + \epsilon)$ and the vertical line
$x = a_- - \epsilon$.
\vs{4}
\ni
We first prove that any solution $z(t) = (x(t),y(t))$ which remains in $\Sigma_\epsilon$ for every
$t \leq t_0$ has to satisfy
$$
\lim_{t \to -\infty} z(t) = (0,0).
$$
Indeed, for such a solution we have that, for $t \leq t_0$,
$$
x(t) \in \left( 0, a_- - \epsilon \right], \qquad x'(t) > 0, \qquad x''(t) > 0,
$$
and, hence, there exist
$$
\lim_{t \to -\infty} x(t) = L < a_- \quad \mbox{ and  } \quad
\lim_{t \to -\infty} x'(t) = 0.
$$
Finally, $L$ must be zero; otherwise:
\[
\liminf_{t \to -\infty}x''(t) = \liminf_{ t \to -\infty } \{ x(t) [ 1 - x(t) ][ a(t) - x(t) ] \} \ge L ( 1 - a_{-} ) \epsilon > 0
\]
and, thus, we would have $\lim_{t \to -\infty} x'(t) \neq 0$.
\vs{6}
\ni
Now we write system \eqref{sysq} as an autonomous system in $\mathbb{R}^3$:
\begin{equation}\label{sysr3}
\left\{
\begin{array}{l}
\vspace{0.1cm}
x' = y \\
\vspace{0.1cm}
y' = x (x-1)(x-q(t)) \\
t' = 1
\end{array}
\right.
\end{equation}
and let $\pi(\cdot;s_0,P)$ be the unique solution
to \eqref{sysr3} starting from $ P \in \mathbb{R}^{3}$ at $s = s_0$. We set
$$
W = \Sigma_\epsilon \times (-\infty,t_0], \quad U = \left(\gamma_- \setminus \{(0,0)\} \right) \times (-\infty,t_0],
\quad
V = \left(\gamma_+ \setminus \{(0,0)\} \right) \times (-\infty,t_0],
$$
where $\gamma_-$ and $\gamma_+$ are the portions of the boundary of $\Sigma_\epsilon$ lying on the
unstable manifolds to $(0,0)$ for the systems $S(a_- -\epsilon)$ and $S(a_+ + \epsilon)$
respectively.
\vs{4}
\ni
Let us study the behavior of the vector field
associated with system \eqref{sysr3} at any point
$P = (x_1,y_1,t_1) \in \partial W$.
First, if $(x_1,y_1) = (0,0)$, it is clear that $\pi(s;t_1,P) = (0,0,s)$ for
all $s$. Next, if $P \in U \cup V$, the vector field
points strictly inwards $W$; therefore,
$\pi(s;t_1,P) \notin W$ for all $s$ in a left neighborhood of $t_1$.
Finally, if either
$$
(x_1,y_1) \in \partial\Sigma_\epsilon \setminus \left( \gamma_- \cup \gamma_+ \cup \{(0,0)\}\right)
$$
or
\[
(x_1,y_1) \in \Sigma_\epsilon \setminus \left( \gamma_- \cup \gamma_+ \cup \{(0,0)\}\right)
\quad \text{and} \quad t_{1}=0
\]
then $\pi(s;t_1,P) \in W$ for all $s$ in a left neighborhood of $t_1$, since in those points the vector field of \eqref{sysr3}
points strictly outwards $W$.
\vs{2}
\ni
Now, we consider the set $D \subset W$ given by
$$
D = \{ P = (x_1,y_1,t_1) \in W : \, \exists s < t_1 \,\mbox{s.t.}\, \pi(s;t_1,P) \notin W \}
$$
and the map
$$
\Phi: D \to \partial W
$$
such that $\Phi(P)$ is the first backward exit point from $W$ for the solution of
\eqref{sysr3} starting from the point $P$, namely $ \Phi(P) = \pi( s^{*}; t_{1}, P ) $ where
\[
s^{*} = \sup \{ s < t_{1} : \pi( s; t_{1}, P ) \not \in D \}.
\]
It is proved in \cite{Con75}
that $\Phi$ is continuous on $D$; moreover, by the previous discussion,
$\Phi(D) = U \cup V$ is not connected and $U$ and $V$ are its connected components.
\vs{4}
\ni
Let $\gamma: [0,1] \to \Sigma_\epsilon$ be a continuous path such that
$\gamma(0) \in \gamma_- \setminus \{ (0,0) \}$ and $\gamma(1) \in \gamma_+ \setminus \{ (0,0) \}$.
Then we have that $\Phi(\gamma(0)) \in U$, $\Phi(\gamma(1)) \in V$.
Since $[0,1]$ is connected, there must be $\tau \in (0,1)$ such that
$\gamma(\tau) \notin D$. By the topological lemma \cite[Corollary 6]{RebZan00}
there exists a continuum $\Gamma_{-\infty}^0 \subset \Sigma_\epsilon \setminus D$
such that
$$
(0,0) \in \Gamma_{-\infty}^0 \quad \mbox{ and } \quad
\Gamma_{-\infty}^0 \cap \left( \{a_- - \epsilon \} \times \mathbb{R} \right) \neq \emptyset.
$$
Letting $\epsilon \to 0^+$, it is actually possible to show that
$\Gamma_{-\infty}^0 \subset \Sigma_0$ (that is, it lies between the unstable manifolds to $(0,0)$ and $(1,0)$
of the systems $S(a_-)$ and $S(a_+)$ respectively) and reaches the vertical line $x = a_-$
(see \cite[Theorem 6, \S 47, II, p.171]{Kur68}).
\vs{4}
\ni
The fact that, in a suitable vertical strip, $\Gamma_{-\infty}^0$ can be written as the graph
of a continuous function can be proved as in \cite[Lemma 2.4]{Ure10}, taking into account the fact
that, for
$$
f(t,x) = x(1-x)(x-q(t)) \quad \mbox{ and } \quad a_-^0 = \frac{1+a_- - \sqrt{a_-^2-a_-+1}}{3},
$$
it holds $\partial_x f(t,x) \geq 0$ for a.e. $t \leq t_0$ and $x \in [0,a_-^0].$
\end{proof}

\begin{remark} \label{stablemanif}
\textnormal{We observe that the mere existence of the stable and unstable manifolds
follows from the Stable Manifold Theorem (see\cite{Hal80}), since $(0,0)$
and $(1,0)$ are hyperbolic equilibrium points of system \eqref{sysq}. Here, we have used Wa\.zewski's method
in order to provide, for such sets, a precise localization, which is indeed needed in our arguments.
Notice that in this way we cannot directly claim
that stable/unstable manifolds are curves (indeed, Wa\.zewski's method applies in more general cases in which this is not true);
however, we can recover such an information in an elementary way, using the sign of the nonlinearity as in \cite{Ure10}.}
\end{remark}

\section{Main results}\label{sec4}
\def\theequation{4.\arabic{equation}}\makeatother
\setcounter{equation}{0}

Let us consider the equation
\begin{equation}\label{eqmain}
\epsilon^2 u'' + f(s,u) = 0
\end{equation}
where
\begin{equation}\label{fmodificata}
f(s,u) = \left\{\begin{array}{ll}
\vspace{0.1cm}
u(1-u)(u-a(s)) & \mbox{if } u \in [0,1],
\\
0 & \mbox{if } u \notin [0,1],
\end{array}
\right.
\end{equation}
with $a: \mathbb{R} \to \mathbb{R}$ a locally integrable function satisfying
$$
0 < a(s) < 1, \quad \mbox{ for every } s \in \mathbb{R}.
$$
\begin{remark}\label{fmodif}
\textnormal{We observe that, in view of the boundedness of $f$, any solution to \eqref{eqmain} is globally defined.
Moreover, if a solution $u$ satisfies $u(s^*) \notin (0,1)$ for some $s^* \in \mathbb{R}$,
then $u(s) \notin (0,1)$ either for all $s \leq s^*$ or for all $s \geq s^*$.}
\end{remark}

From now on, we assume
\begin{equation} \label{definizionea}
a(s) = \left\{\begin{array}{ll}
\vspace{0.1cm}
a_- & {\mbox{if $s_{2k} \leq s < s_{2k+1}$}}
\\
a_+ & {\mbox{if $s_{2k+1} \leq s < s_{2k+2}$ }},
\end{array}
\right.
\end{equation}
where $a_-,a_+$ are real constants with
\begin{equation} \label{condizionea}
0 < a_- < \frac{1}{2} < a_+ < 1
\end{equation}
and $(s_k)_{k \in \mathbb{Z}} \subset \mathbb{R}$ is a sequence such that:
\begin{equation}\label{lowbound}
0 < \delta \leq \inf_{ k \in \mathbb{Z} } ( s_{k+1} - s_k )
\end{equation}
and, thus, $s_k < s_{k+1}$ for every $k\in \mathbb{Z}$ and $\lim_{k \to \pm \infty}s_k = \pm \infty$.
For every $j\in \mathbb{Z}$ let us set
$$
I_j = [s_{6(j-1)},s_{6j}], \quad I_j^+ = [s_{6j-5},s_{6j-4}], \quad I_j^- = [s_{6j-2},s_{6j-1}].
$$
In this setting, we will prove the following results.
\smallbreak
The first one, Theorem \ref{chaos},
deals with chaotic solutions.

\begin{Theorem}\label{chaos}
For any integer $M \geq 1$, there exists $\epsilon^* = \epsilon^*(M) > 0$,
such that, for any $\epsilon \in (0,\epsilon^*)$ and for any
double infinite sequence ${\bf n} = \{(n_j^{+},n_j^-)\}_{j\in \mathbb{Z}}$, with $ n_j^{\pm} \in \{ 1, \dots, M \} $
for all $ j \in \mathbb{Z} $, there exists a globally defined solution $ u $ of \eqref{eqmain},
with $0 < u(s) < 1$ for all $s \in \mathbb{R}$, such that its trajectory
$ ( u, u' ) $ makes $n_j^{\pm} $ turns around $ ( a_{\pm}, 0 ) $ in the interval $ I_j^{\pm} $, for all $ j \in \mathbb{Z} $.
Moreover, if the sequence $ ( s_k - s_{k-1} ) $ is $ 6 $-periodic
and the sequence $ {\bf n}$ is $ \ell $-periodic for some $\ell \in \mathbb{N}$, then this solution $ u $ is
$\ell(s_6-s_0)$-periodic.
\end{Theorem}

Notice that the sentence ``$n_j^{\pm}$ turns around $ ( a_{\pm}, 0 ) $'' in the above statement has to be meant according to Remark \ref{nodali};
the same terminology will be used in Theorems \ref{eter} and \ref{omo} below, dealing with heteroclinic and homoclinic solutions, respectively.

\begin{Theorem}\label{eter}
For any integer $M \geq 1,$ there exists $\epsilon^* = \epsilon^*(M) > 0$,
such that, for any $\epsilon \in (0,\epsilon^*)$, for any integer
$K \geq 0$ and for any ${\bf n} = \{(n_j^{+},n_j^-)\}_{j=1,\ldots,K}$ (this $K$-uple is considered to be empty
if $K = 0$) with  $ n_j^{\pm} \in \{ 1, \dots, M \} $
for all $ j =1,\ldots,k$, there exists a globally defined solution $u$ of
\eqref{eqmain}, with $0 < u(s) < 1$ for all $s \in \mathbb{R}$, such that:
\begin{enumerate}
\item $(u(-\infty),u'(-\infty)) = (0,0)$,
\item $(u(+\infty),u'(+\infty)) = (1,0)$,
\item the trajectory
$ ( u, u' ) $ makes $n_j^{\pm} $ turns around $ ( a_{\pm}, 0 ) $ in the interval $ I_j^{\pm} $, for all $ j =1,\ldots,K $,
\item $u$ is monotone in $(-\infty,s_0]$ and in $[s_{6K},+\infty)$.
\end{enumerate}
\end{Theorem}

\begin{Theorem}\label{omo}
For any integer $M \geq 1,$ there exists $\epsilon^* = \epsilon^*(M) > 0$,
such that, for any $\epsilon \in (0,\epsilon^*)$, for any integer
$K \geq 0$ and for any ${\bf n} = \{(n_j^{+},n_j^-)\}_{j=1,\ldots,K}$ (this $K$-uple is considered to be empty
if $K = 0$) with  $ n_j^{\pm} \in \{ 1, \dots, M \} $
for all $ j =1,\ldots,k$, there exists a globally defined solution $u$ of
\eqref{eqmain},
with $0 < u(s) < 1$ for all $s \in \mathbb{R}$, such that:
\begin{enumerate}
\item $(u(-\infty),u'(-\infty)) = (0,0)$,
\item $(u(+\infty),u'(+\infty)) = (0,0)$,
\item the trajectory
$ ( u, u' ) $ makes $n_j^{\pm} $ turns around $ ( a_{\pm}, 0 ) $ in the interval $ I_j^{\pm} $, for all $ j =1,\ldots,K $,
\item $u$ is monotone in $(-\infty,s_0]$ and in $[s_{6K+1},+\infty)$ and $u'$ vanishes exactly once
in $(s_{6K},s_{6K+1})$.
\end{enumerate}
\end{Theorem}


\subsection{Proof of the main results}

In order to prove the above theorems, we perform the change of variable
$$
x(t) = u(\epsilon t), \quad t \in \mathbb{R},
$$
converting the equation \eqref{eqmain} into
\[
x'' + f(\epsilon t , x ) = 0,
\]
with
$$
f(\epsilon t, x ) = \begin{cases}
x (1-x)(x-a_{-} ) & \text{if } x \in [ 0, 1 ] \text{ and } t \in \left[ t_{2k}, t_{2k+1} \right), k \in \mathbb{Z} \\
x (1-x)(x-a_{+} ) & \text{if } x \in [ 0, 1 ] \text{ and } t \in \left[ t_{2k+1}, t_{2k+2} \right), k \in \mathbb{Z} \\
0                 & \text{if } x \not\in [ 0, 1 ]
\end{cases}
$$
where we have set
$$
t_k = \frac{s_k}{\epsilon}, \quad \mbox{ for every } k \in \mathbb{Z}.
$$
Notice that, in view of \eqref{lowbound}, we now have
$$
0 < \frac{\delta}{\epsilon} \leq \inf_{k\in\mathbb{Z}} ( t_{k+1} - t_k ).
$$
Let us define, for any $k \in \mathbb{Z}$
$$
\mathcal{D}_k = \mathcal{P}_k = \mathcal{R}_1 \quad \mbox{ and } \quad \widetilde{\mathcal{P}_k} = (\mathcal{R}_1,\mathcal{R}_1^-)
$$
and
\begin{align*}
\phi_k = & \, \Psi_+^{t_{6k+6}-t_{6k+5}}   \circ \Psi_-^{t_{6k+5}-t_{6k+4}} \circ
\Psi_+^{t_{6k+4}-t_{6k+3}} \circ \Psi_-^{t_{6k+3}-t_{6k+2}} \circ \\
& \circ \Psi_+^{t_{6k+2}-t_{6k+1}}  \circ \Psi_-^{t_{6k+1}-t_{6k}}.
\end{align*}
In other words, $ \Phi_{k}(p) = ( x( t_{6k+6} ), x'( t_{6k+6} ) ) $ where $ x $ is the solution of $ x'' + f( \epsilon t, x ) = 0 $
with $ ( x( t_{6k} ), x'( t_{6k} ) ) = p $.
\begin{proof}[Proof of Theorem \ref{chaos}]
Given $M \geq 1$, let us define $\epsilon^*(M) > 0$ as
\begin{equation}\label{epsm}
\epsilon^*(M)  = \frac{\delta}{\max\{T_1^*,T_2^*(M)\}},
\end{equation}
where $T_1^*,T_2^*(M)$ are given by Proposition \ref{stretching}.
Now we take $\epsilon \in (0,\epsilon^*)$; notice that we have
$$
t_{k+1} - t_k > \max\{T_1^*,T_2^*(M)\}, \quad \mbox{ for every } k\in \mathbb{Z}.
$$
In view of Propositions \eqref{stretching} and \eqref{composizione}, we have
$$
\phi_k: \widetilde{\mathcal{P}_k} \,\stretchx^{N^2}\, \widetilde{\mathcal{P}_k}.
$$
The conclusion thus follows from Theorem \ref{thchaos} in the Appendix. We finally notice that
$x(t_{6k}) \in (0,1)$ for all $k \in \mathbb{Z}$, so that, in view of Remark \ref{fmodif},
$x(t) \in (0,1)$ for all $t \in \mathbb{R}$.
\end{proof}

\begin{proof}[Proof of Theorem \ref{eter}]
We begin by applying Proposition \ref{waze} on the intervals $ \left( -\infty, t_{-1} \right] $
and $ \left[ t_{6k+1}, +\infty \right) $
in order to find the continuous functions $ y_{-\infty}^{0} : [ 0, a_{-}^{0} ] \to \mathbb{R}^{+} $ and
$ y_{+\infty}^{1} : [ a_{+}^{1}, 1 ] \to \mathbb{R}^{+} $ such that $ ( x, y_{-\infty}^{0}(x) ) \in \Gamma_{-\infty}^{0} $ and
$ ( x, y_{+\infty}^{1}(x) ) \in \Gamma_{+\infty}^{1} $ for all suitable $ x $.
We denote by $x^*$ the abscissa of the intersection point in the first quadrant between the
homoclinic to $(0,0)$ for the system $S(a_-)$ and the orbit $\Theta(a_+,(p_+,0))$ of system
$S(a^+)$. Let
$$
x_1 = \min\{x^*,a_-^0\}.
$$
The time needed by a solution of system $S(a_+)$ to run along
$\Theta(a_+,(p_+,0))$ from $x=x_1$ to $x=1$ is given by
$$
\tau = \frac{1}{\sqrt{2}}\int_{x_1}^1 \frac{dx}{\sqrt{F_{a_+}(p_+)-F_{a_+}(x)}}.
$$
In a similar way, we take $x^{**}$ as the abscissa of the intersection point in the first quadrant between the
homoclinic to $(1,0)$ for the system $S(a_+)$ and the orbit $\Theta(a_-,(p^-,0))$ of system
$S(a^-)$. Let
$$
x_2 = \max\{x^{**},a_+^1\}.
$$
The time needed by a solution of system $S(a_-)$ to go on
$\Theta(a_-,(p^-,0))$ from $x=0$ to $x=x_2$ is given by
$$
\tau' = \frac{1}{\sqrt{2}}\int_{0}^{x_2} \frac{dx}{\sqrt{F_{a_-}(p^-)-F_{a_-}(x)}}.
$$
For any integer $M \geq 1$, we can thus define
\begin{equation}\label{epsm2}
\epsilon^*(M)  = \frac{\delta}{\max\{T_1^*,T_2^*(M),\tau,\tau'\}},
\end{equation}
where again $ T_{1}^{*} $ and $ T_{2}^{*}(M) $ are given by Proposition \ref{stretching}.
With these positions, we have that
$$
\Psi_+^{t_0-t_{-1}}(\Gamma_{-\infty}^0 \cap ([0,x_1] \times \mathbb{R}))
$$
is a path which crosses $\mathcal{R}_1$ connecting the two components of $\mathcal{R}_1^-$.
\vs{2}
\ni
Indeed, this is a consequence of the fact that $(0,0)$ is a fixed point for
$\Psi_+^{t_0-t_{-1}}$ and that the point $(x_1,y_1) \in \Gamma_{-\infty}^0$ is mapped
by $\Psi_+^{t_0-t_{-1}}$ in the half-plane $\{ x \geq 1\}$ which follows from the choice of $\tau$ and the monotonicity of the map
\[
c \mapsto \int_{x_1}^1 \frac{dx}{\sqrt{c-F_{a_+}(x)}}.
\]
Analogously (by the choice of $\tau'$)
$$
(\Psi_-^{t_{6k+1}-t_{6k}})^{-1}(\Gamma_{+\infty}^1 \cap ([0,x_1] \times \mathbb{R}))
$$
is a path which crosses $\mathcal{R}_1$ connecting the two components of $\mathcal{R}_1^+$.
\vs{4}
\ni
Now, corresponding to the choice of ${\bf n} = \{(n_j^{+},n_j^-)\}_{j=1,\ldots,K}$ with $ n_j^{\pm} \in \{ 1, \dots, M \} $,
there exists a sub-path $\gamma$ of $\Psi_+^{t_0-t_{-1}}(\Gamma_{-\infty}^0)$ which is stretched across
$(\mathcal{R}_1,\mathcal{R}_1^-)$ by the map $\phi_k \circ \phi_{k-1} \circ \ldots \circ \phi_1$.
\vs{4}
\ni
By the topological lemma \cite[Lemma 3]{MulWil74} the intersection
$$
(\phi_k \circ \phi_{k-1} \circ \ldots \circ \phi_1)(\gamma) \cap (\Psi_-^{t_{6k+1}-t_{6k}})^{-1}(\Gamma_{+\infty}^1)
$$
is not empty and any point in this intersection gives rise to the required solution.
\end{proof}

\begin{proof}[Proof of Theorem \ref{omo}]
One can argue in a similar way as in the proof of Theorem \ref{eter},
using here the curves $\Gamma_{-\infty}^0$ and $\Gamma_{+\infty}^0$ given in Lemma \ref{waze}.
\end{proof}

\begin{remark}\label{stabile}
\textnormal{We notice that the statements of our results can be modified (in some cases) to cover also
more general equations $\epsilon^2 u'' + u(1-u)(u-\tilde{a}(t)) = 0$ with $\tilde a$ close to a stepwise function. For instance,
in the setting of Theorems \ref{eter} and \ref{omo} and given integers $M, K$ and $\epsilon \in (0,\epsilon^*(M))$,
the existence result still holds for all functions $\tilde{a}$ with $L^1$-norm on $[0,s_{6K}]$ smaller than
a constant $\delta = \delta(M,K,\epsilon)$ (compare with \cite[Remark 4.1]{MarRebZan10} for more details on
the stability of the stretching technique, and recall that Lemma \ref{waze} about the existence of stable/unstable manifolds is
indeed proved for a more general, non-stepwise, function).
We also observe that the value $\epsilon^*(M)$ is an explicit constant (see \eqref{epsm} and \eqref{epsm2}).}
\end{remark}

\begin{remark}\label{linkinginterno}
\textnormal{We collect here some hints for possible variants and generalizations of our results.
\begin{itemize}
\item The role of the equilibria $(0,0)$ and $(1,0)$ may be switched in Theorems \ref{eter} and \ref{omo}.
More precisely, we can provide also solutions homoclinic to $(1,0)$ as well as heteroclinic solutions
converging to $(1,0)$ for $t \to -\infty$ and to $(0,0)$ for $t \to +\infty$. Moreover, the arguments used in the proof
of Theorem \ref{eter} can be easily modified in order to find multiple solutions for Sturm-Liouville like boundary values problems
(e.g., Dirichlet and Neumann ones).
\item According to the discussion in Remark \ref{LTM}, it is possible to produce a further family of chaotic solutions
when condition \eqref{link} is satisfied.
\item In our main results we have considered functions of the form \eqref{definizionea} with $a_-$ and $a_+$ satisfying condition \eqref{condizionea},
suggested by the investigation in \cite{AngMalPel87}. When either $a_+<1/2$ or $a_->1/2$, the superposition
of the phase-portraits of the systems $S(a_+)$ and $S(a_-)$ gives rise to a different configuration (see Figure \ref{zaza})
analogous to the one considered in the paper \cite{ZaZa-14}, where a Schr\"{o}dinger equation is studied. As a consequence,
by combining the arguments therein together with Lemma \ref{waze},
it is possible to obtain the existence of chaotic dynamics and homoclinic orbits.
\end{itemize}}
\end{remark}

\begin{figure}[!h]
\centering
\includegraphics[height=8cm,width=11cm]{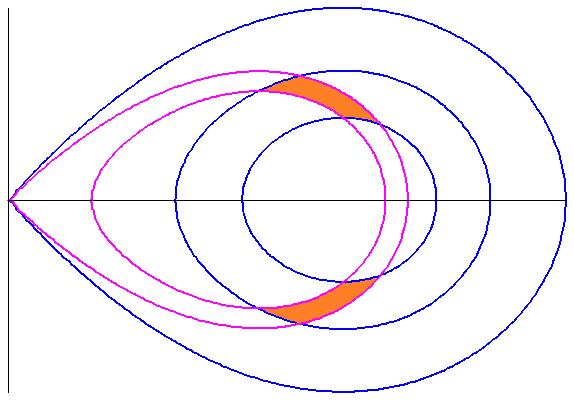}
\caption{\small{The superposition
of the phase-portraits of the systems $S(a_-)$ and $S(a_+)$, with $a_- = 0.3$ and
$a_+ = 0.4$. Orbits of the system $S(a_-)$ are painted in pink, orbits of the system $S(a_+)$ are painted in blue,
topological rectangles verifying stretching properties are painted in orange.
Compare with Figure 2 in \cite{ZaZa-14}.}}
\label{zaza}
\end{figure}

\section{Appendix: SAP method}
\def\theequation{5.\arabic{equation}}
\makeatother\setcounter{equation}{0}

In this appendix we collect the definitions and results on the Stretching Along Paths method which are needed in our paper.
We refer to \cite{Bur14,PasPirZan08} for a comprehensive presentation of the theory and further references.
\medbreak
By a path $\gamma$ in
$\mathbb{R}^2$ we mean a continuous mapping $\gamma:[0,1] \to
\mathbb{R}^2$, while by a sub-path $\sigma$ of $\gamma$ we just
mean the restriction of $\gamma$ to a compact subinterval of
$[0,1]$. By an \emph{oriented rectangle} we mean a pair
$$
\mathcal{\widetilde R} = (\mathcal{R},\mathcal{R}^-),
$$
being $\mathcal{R} \subset \mathbb{R}^2$ homeomorphic to $[0,1]^2$
(namely, a \textit{topological rectangle}) and
$$
\mathcal{R}^- = \mathcal{R}^-_1 \cup \mathcal{R}^-_2
$$
the disjoint union of two compact arcs (by definition, a compact
arc is a homeomorphic image of $[0,1]$) $\mathcal{R}^-_1,
\mathcal{R}^-_2 \subset \partial \mathcal{R}$.
\smallbreak
With these preliminaries, we can give the following definition.

\begin{Definition}\label{stretchdef}
Let $\mathcal{\widetilde A} = (\mathcal{A},\mathcal{A}^-)$,
$\mathcal{\widetilde B} = (\mathcal{B},\mathcal{B}^-)$ be oriented
rectangles and let $\Psi: \mathcal{D}_{\Psi} \subset \mathbb{R}^2 \to \mathbb{R}^2$ be a continuous map.
\begin{itemize}
\item[-] We say that $(\mathcal{H},\Psi)$ stretches
$\mathcal{\widetilde A}$ to $\mathcal{\widetilde B}$ along the
paths and write
$$
(\mathcal{H},\Psi): \mathcal{\widetilde A} \, \stretchx \,
\mathcal{\widetilde B}
$$
if $\mathcal{H} \subset \mathcal{A} \cap \mathcal{D}_{\Psi}$ is a compact subset and for every path $\gamma: [0,1] \to \mathcal{A}$ such that
$\gamma(0) \in \mathcal{A}^-_1$ and $\gamma(1) \in
\mathcal{A}^-_2$ (or $\gamma(0) \in \mathcal{A}^-_2$ and
$\gamma(1) \in \mathcal{A}^-_1$), there exists a subinterval
$[t',t''] \subset [0,1]$ such that for every $t \in [t',t'']$
$$
\gamma(t) \in \mathcal{H}, \qquad \Psi(\gamma(t)) \in \mathcal{B},
$$
and, moreover, $\Psi(\gamma(t'))$ and $\Psi(\gamma(t''))$ belong
to different components of $\mathcal{B}^-$.
\item[-] We say that $\Psi$ stretches $\mathcal{\widetilde A}$ to $\mathcal{\widetilde B}$ along the
paths with crossing number $M \geq 1$ and write
$$
\Psi: \mathcal{\widetilde A} \,\stretchx^M \,
\mathcal{\widetilde B}
$$
if there exist $M$ pairwise disjoint compact sets $\mathcal{H}_1,\ldots,\mathcal{H}_M \subset \mathcal{A} \cap \mathcal{D}_{\Psi}$
such that $(\mathcal{H}_i,\Psi): \mathcal{\widetilde A} \stretchx
\mathcal{\widetilde B}$ for $i=1,\ldots,M$.
\end{itemize}
\end{Definition}

As an easy consequence of the definition, we have that the stretching property has a good behavior with respect to compositions of maps.

\begin{proposition}\label{composizione}
Assume that $\Psi$ stretches $\mathcal{\widetilde A}$ to $\mathcal{\widetilde B}$
with crossing number $M$ and $\Phi$ stretches $\mathcal{\widetilde B}$ to $\mathcal{\widetilde C}$
with crossing number $N$. Then, the composition $\Phi \circ \Psi$ stretches $\mathcal{\widetilde A}$
to $\mathcal{\widetilde C}$ with crossing number $M \times N$.
\end{proposition}

We finally state the result which is employed in our paper.

\begin{Theorem}\label{thchaos}
Assume that there are double sequences of oriented rectangles
$\left(\mathcal{\widetilde P}_k \right)_{k \in \mathbb{Z}}$ and of maps
$\left(\phi_k \right)_{k \in \mathbb{Z}}$ such that $\phi_k$ stretches
$\mathcal{\widetilde P}_k$ to $\mathcal{\widetilde P}_{k+1}$ with crossing number $M_k \geq 1$,
for all $k \in \mathbb{Z}$.
Let $\mathcal{H}_{k,j} \subset \mathcal{P}_{k}$, with $j=1,\ldots,M_k$, be the compact sets according to the definition of multiple stretching.
Then, the following conclusion hold:
\begin{itemize}
\item for every sequence $(s_k)_{k \in \mathbb{Z}}$, with $s_k \in \{1,\ldots,M_k\}$, there exists a sequence $(w_k)_{k \in \mathbb{Z}}$
with $w_k \in \mathcal{H}_{k,s_k}$ and $\phi_k(w_k) = w_{k+1}$ for all $k \in \mathbb{Z}$;
\item if there exists $h,k \in \mathbb{Z}$ with $h <k$ such that $\mathcal{\widetilde P}_h = \mathcal{\widetilde P}_k$,
then there is a finite sequence $(w_i)_{h \leq i \leq k}$ with $w_i \in \mathcal{H}_{i,s_i}$,
$\phi_i(w_i) = w_{i+1}$ for $i=h,\ldots,k-1$ and $w_h = w_k$, that is, $w_h$ is a fixed point
of $\phi_{k-1} \circ \ldots \circ \phi_h$.
\end{itemize}
\end{Theorem}

\begin{proof}
By assumption, we have that
$$
\left(\mathcal{H}_{k,s_k},\phi_k \right) : \mathcal{\widetilde P}_{k}\,\stretchx \, \mathcal{\widetilde P}_{k+1}, \quad \mbox{ for every } k \in \mathbb{Z}.
$$
The conclusion follows then from \cite[Theorem 2.2]{PapZan04}.
\end{proof}

\begin{remark}\label{coniugio}
\textnormal{In the setting of Theorem \ref{thchaos}, when
$\mathcal{\widetilde P}_k = \widetilde{\mathcal{P}}$, $\phi_k = \phi$ and
$M_k = M$ for all $k \in \mathbb{Z}$, it is possible to prove
that there is a compact invariant set $\Lambda \subset \mathcal{P}$
such that the map $\phi|_{\Lambda}$ is topologically semiconjugate on
the Bernoulli shift on $M$ symbols (see \cite[Lemma 2.3]{PasPirZan08}).}
\end{remark}

\end{document}